\documentclass[11pt]{amsart}
 \usepackage{amsmath,amssymb,enumerate}
 \usepackage[all]{xy}
 \usepackage{xcolor} 
 
 \newtheorem{pro}{Proposition}[section]
 \newtheorem{cor}[pro]{Corollary}
 \newtheorem{lem}[pro]{Lemma}
 \newtheorem{rem}[pro]{Remark}
 \newtheorem{exa}[pro]{Example}
  \newtheorem{defi}[pro]{Definition}
 \newtheorem{prop}[pro]{Proposition}
 \newtheorem{thm}[pro]{Theorem}

 \newcommand{\h}{\mathcal H}

 \numberwithin{equation}{section}

\usepackage{hyperref}

\begin{document}

  \title[  Controlled $K$-Fusion Frame for Hilbert Spaces]{Controlled $K$-Fusion Frame for Hilbert Spaces}

\setcounter{tocdepth}{1}

  \author{N. ASSILA, S. KABBAJ and B. Moalige} 

 \address{Faculty of Sciences  \\  Ibn Tofail University \\
Kenitra \\
Marocco}

\email{nadia.assila@uit.ac.ma}

\address{Faculty of Sciences  \\  Ibn Tofail University \\
Kenitra \\
Marocco}

\email{samkabbaj@yahoo.fr}
\address{Faculty of Sciences  \\  Ibn Tofail University \\
Kenitra \\
Marocco}

\email{brahim.moalige@uit.ac.ma}

 \date{\today 
 }

 \begin{abstract}
$K$-fusion frames are a generalization of fusion frames in frame theory. In this paper, we extend the concept of controlled fusion frames to controlled $K$-fusion frames, and we develop some results on the controlled $K$-fusion frames for Hilbert spaces, which generalized some well known  of controlled fusion frames case. also we discuss some characterizations of controlled Bessel $K$-fusion sequences and of controlled Bessel $K$-fusion. Further, we analyse stability conditions of controlled $K$-fusion frames under perturbation.

 \end{abstract}
 
 \maketitle

\section{Introduction}
Fames are more flexible than bases to solve some problems in Hilbert spaces. They were firstly introduced by Duffin and Schaeffer \cite{DS52} to study nonharmonic Fourier series in 1952, and widely studied by Daubechies, Grossman and Meyer \cite{DGM86} in 1986. More results of frames are in \cite{Ch16}.\\
Fusion frames as a genaralisation of frames were introduced by Casazza and Kutyniok in \cite{CK04} and further there were developed in their joint paper \cite{CKL08} with Li. The theory for fusion  frames is available in arbitrary separable Hilbert spaces (finite-dimensional or not). The motivation behind fusion frames comes from signal processing, more precisely, the desire to process and analyze large data sets efficiently. A natural idea is to split such data sets into suitable smaller "blocks" which can be treated independently. From a pure mathematical point of view, fusion frames are special cases of the $g$-frames \cite{SW06}. However, the connection to concrete applications is less apparent from the more abstract definition of $g$-frames. In 2012, L. Gavruta \cite{Ga12}  introduced the notions of $K$-frames in Hilbert space to study the atomic systems with respect to a bounded linear operator $K$. Controlled frames in Hilbert spaces have been introduced by P. Balaz \cite{ABG10} to improve the numerical efficiency of iterative algorithms for inventing  the frame operator. Further A. Khosravi \cite{KM12} generalized this concept to the case of fusion frames. He has showed that controlled fusion framee as a generalization of fusion frames give a generalized way to obtain numerical advantage in the sense of preconditioning to check the fusion frame condition. In  2015 Rahimi \cite{NNR15} defined the concept of controlled $K$-frames in Hilbert spaces and showed that controlled $K$- frames are equivalent to $K$-frames. \\  
Motivated by the above literature, we introduce and investigate some properties of controlled $K$-fusion frames, we also generalize some known results for controlled fusion frames to controlled $K$-fusion frames. Finaly, we  present  perturbation result for controlled $K$-fusion frames.  \\
This paper is organized as follows. In Section 2, we recall several definitions about fusion frames, $K$-fusion frames and controlled fusion frames. Then, we give a basic properties about a bounded linear operator. In Section 3, we introduce the concept of controlled $K$-fusion frames and discuss their properties. In section 4, we analyze stability conditions of controlled $K$-fusion frames under perturbation.
\section{Preliminaries and Notations}
 Throughout this paper, we will adopt the following notations. $\mathcal{H}$ is separable Hilbert space, $\{W_i\}_{i\in I}$ is sequence of closed subspaces of $\mathcal{H}$, where $I$ is a countable index set. the family of all bounded linear operators for $\mathcal{H}$ into $\mathcal{H}$ is denoted $B(\mathcal{H})$. We
denote $\mathcal{R}_T$, $\mathcal{N}_T$, range and null space  of a bounded linear operator $T$, respectively. $GL(\h)$ is the set of all bounded invertible operators on $\h$ with bounded inverse, and $GL(\h)^+$ denotes the set of all positive operators in $GL(\h)$. $\pi_{W_i}$ is the orthogonal projection from $\mathcal{H}$ into $W_i$, and $\{w_i\}_{i\in I}$ is a family of weights, i.e. $w_i>0$, for any $i\in I$.\\
The space $(\oplus_{i\in I}\mathcal{H})_{l^2}$ is defined by \begin{align*}
(\oplus_{i\in I}\mathcal{H})_{l^2}=\{\{f_i\}_{i\in I}: f_i\in\mathcal{H}, i\in I, \sum_{i\in I}\Vert f_i\Vert^2<\infty  \},
\end{align*}
with the inner product  is defined by
$$\langle \{f_i\}_{i\in I},\{g_i\}_{i\in I}\rangle= \sum_{i\in I} \langle f_i, g_i\rangle_{\mathcal{H}}.$$
 $(\oplus_{i\in I}\mathcal{H})_{l^2}$ is a separable Hilbert space \cite{KM12}.
 \subsection{Preliminaries}
 \subsection{Fusion frames}
  \begin{defi}\cite{CKL08}
let $\{W_i\}_{i\in I}$ be a family of closed subspaces of a Hilbert space $\mathcal{H}$. let $\{w_i\}_{i\in I}$ be a family of weights, the family $\mathcal{W}=\{W_i,w\}_{i\in I}$ is called a $K-$fusion frame for $\mathcal{H}$, if there exist positive constants $A\leq B<\infty$ such that \begin{align}\label{def: fusionframe}
 A\Vert f\Vert^2\leq \sum_{i\in I} w_i^2\Vert \pi_{W_i}f\Vert^2\leq B\Vert f\Vert^2,  f \in\mathcal{H}.
 \end{align}
 \end{defi}
 $A$ and $B$ are called lower and upper bounds of fusion frame, respectively.\\  
 If only the right inequality of \ref{def: fusionframe} holds, we call the family $\{W_i,w\}_{i\in I}$ is  fusion bessel sequence. 
 \subsection{$K-$fusion frames}

\begin{defi}\cite{AN18}
 Let $K\in B(\mathcal{H})$, let $\{W_i\}_{i\in I}$ be a family of closed subspaces of a Hilbert space $\mathcal{H}$, and let $\{w_i\}_{i\in I}$ be a family of weights. Then the family $\mathcal{W}=\{W_i,w\}_{i\in I}$ is called a $K-$fusion frame for $\mathcal{H}$, if there exist positive constants $A\leq B<\infty$ such that \begin{align}
 A\Vert K^*f\Vert^2\leq \sum_{i\in I} w_i^2\Vert \pi_{W_i}f\Vert^2\leq B\Vert f\Vert^2,  f \in\mathcal{H}.
 \end{align}
 Where $K^*$ is the adjoint operator of $K$.
 \end{defi}
 $A$ and $B$ are called lower and upper bounds of $K$-fusion frame, respectively.  
 
 suppose that $\{W_i,w_i\}_{i\in I}$ is a fusion Bessel sequence for $\h$, then the synthesis operator of $\{W_i,w_i\}_{i\in I}$ is defined by $T_W:\quad (\sum_{i\in I}\oplus W_i)_{l^2}\longrightarrow \h$,
 \begin{align*}
 T_{W}(\{f_i\}_{i\in I})=\sum_{i\in I} w_i f_i,\quad \{f_i\}_{i\in I}\in(\sum_{i\in I}\oplus W_i)_{l^2}. 
 \end{align*}
 Where  \begin{align*}
((\sum_{i\in I}\oplus W_i)_{l^2}=\{\{f_i\}_{i\in I}: f_i\in W_i, i\in I, \sum_{i\in I}\Vert f_i\Vert^2<\infty  \}.
\end{align*}
  Its adjoint operator, whitch is called the analysis operator  $T^*_W:\quad \h \longrightarrow(\sum_{i\in I}\oplus W_i)_{l^2} $, is defined by 
\begin{align*}
T^*_W(f)=\{ w_i \pi_{W_i} f\}_{i\in I},\quad f\in \h.
\end{align*}
And the $K$-fusion frame operator associated is $S_W:\quad\h \longrightarrow \h$.\begin{align}\label{sfusionframeoper}
S_W(f)=\sum_{i\in I} w_i^2 \pi_{W_i}f, \quad f\in \h.
\end{align}
\subsection{controlled fusion frame}
\begin{defi}\cite{KM12}
Let $\{W_i\}_{i\in I}$ be a family of closed subspaces of a Hilbert space $\mathcal{H}$, let $\{w_i\}_{i\in I}$ be a family of weights, and let $T,U \in GL(\h)$ . Then the family $\mathcal{W}=\{W_i,w\}_{i\in I}$ is called a $(T,U)$-controlled fusion frame for $\mathcal{H}$, if there exist positive constants $A\leq B<\infty$ such that \begin{align}
 A\Vert K^*f\Vert^2\leq \sum_{i\in I} w_i^2 \langle \pi_{W_i}Tf,\pi_{W_i}Uf\rangle\leq B\Vert f\Vert^2,  f \in\mathcal{H}.
 \end{align}
 \end{defi}
 $A$ and $B$ are called lower and upper bounds of $(T,U)$-controlled fusion frame, respectively. 
For further information in $K$-fusion frame and controlled fusion frame theory we refer the reader to \cite{AN18} and  \cite{KM12}.\\
  In theory of frames, often use the following theorem, which describes some properties of the adjoint operator.
 \begin{thm}\cite{Ch16}\label{lem: adjoint}
 Let $\mathcal{H}_1,\mathcal{H}_2$ be Hilbert spaces, and suppose that  $U \in B(\mathcal{H}_1,\mathcal{H}_2)$. Then, 
 \begin{itemize}
 \item [i)] $U^* \in B(\mathcal{H}_2,\mathcal{H}_1)$ and $\Vert U^*\Vert=\Vert U\Vert$.
 \item [ii)] $U$ is surjective if and only if $\exists A >0$ such that $\Vert U^*h \Vert_{\h_2}\geq A\Vert h\Vert_{\h_1}$.
 \end{itemize}
 \end{thm}
 It is well-Known that not all bounded operator  $U$ on a Hilbert space $\h$ are invertible: an operator $U$ needs to be injective and  surjective in order to be invertible. For doing this, one can use right-inverse operator. The following lemma shows that if an operator $U$ has closed range, there exists a "right-inverse operator" $U^\dagger$ in the following sense:
 \begin{lem}\cite{Ch16}\label{lem: Pseudo-inverse}
 Let $\mathcal{H}_1,\mathcal{H}_2$ be Hilbert spaces, and suppose that  $U \in B(\mathcal{H}_1,\mathcal{H}_2)$ with closed range $\mathcal{R}_U$. then there exists a bounded operator $U^\dagger : \quad \mathcal{H}_2\longrightarrow\mathcal{H}_1 $ for which \begin{align}
 U U^\dagger x=x, \quad  x \in \mathcal{R}_U, 
\end{align}
 and \begin{align*}
 (U^*)^\dagger=(U^\dagger)^*.
 \end{align*}
 \end{lem}
 The operator $U^\dagger$ is called the Pseudo-inverse of $U$.\\
 In the literature, one will ofen see the pseudo-inverse of an operator $U$ with closed range defined as the unique operator $U^\dagger$ satisfying that \begin{align*}
 \mathcal{N}_{U^\dagger}=\mathcal{R}_U^\perp,\quad U U^\dagger x=x, \quad  x \in \mathcal{R}_U.
\end{align*}  
The following lemma is necessary for our results.
 \begin{lem}\cite{Ga07}\label{lem: Projection-operator}
Let $V \subseteq \mathcal{H}$ be a closed subspace, and $T$ be a linear bounded operator on
$\mathcal{H}$. Then
\begin{align}
\label{lem: pi-T}
\pi_{V} T^*=\pi_ V T^*\pi_{\overline{{T V}}}.
\end{align}
If $T$ is a unitary (i.e. $T^*T = T T^*=Id_{\mathcal{H}}$, then
\begin{align}
\label{lem: piTunitaire}
\pi_{\overline{{TV}}} T=T\pi_V .
\end{align} \end{lem}
\begin{pro}\cite{Pa73} \label{prop: mmMM}
 Let $T: \quad \mathcal{H} \longrightarrow \mathcal{H}$ be a linear operator. Then the following condition are equivalent:
 \begin{enumerate}
 \item There exist $m > 0$ and $M <\infty $, such that $m I \leq T \leq M I$;
\item $T$ is positive and there exist $m > 0$ and $M < \infty$, such that $m\Vert f \Vert^2\leq \Vert T^{\frac{1}{2}}f\Vert^2\leq M\Vert f \Vert^2$ for all $f\in \mathcal{H}$;
\item $T$ is positive and $T^{\frac{1}{2}}\in GL(\mathcal{H})$;
\item There exists a self-adjoint operator $A \in GL(\mathcal{H})$ , such that
$A^2=T$;
\item $T \in GL^+(\mathcal{H})$.
 \end{enumerate}
 \end{pro} 
 The following lemma will be used in the sequel.
 \begin{lem}\cite{FYY09}\label{lem: orthogonally complemented}
 Let $\mathcal{F}$, $\mathcal{G}$, $\mathcal{H}$ be Hilbert spaces. Let $T\in B(\mathcal{F},\mathcal{G})$ and $T'\in  B(\mathcal{H},\mathcal{G})$ with $\overline{\mathcal{R}}_{T^*}$ be orthogonally complemented.\\ Then the following statement are equivalent:
 \begin{itemize}
 \item[i)] $T'T'^*\leq \lambda T T^*$ for some $\lambda>0$.
 \item[ii)] There exists $\mu>0$ such that $\Vert T'^*z\Vert \leq \mu \Vert T^*z\Vert$, for all $z\in \mathcal{G}$.
 \end{itemize}
 \end{lem}
  
\section{Controlled $K$-fusion frame}
In this section, we introduce the notion of Controlled $K$-fusion frames in Hilbert spaces and we discuss some their properties.\begin{defi} Let $K\in B(\mathcal{H})$, let $\{W_i\}_{i\in I}$ be a family of closed subspaces of a Hilbert space $\mathcal{H}$, let $\{w_i\}_{i\in I}$ be a family of weights, and let $C, C'\in GL(\mathcal{H})$. Then  $\mathcal{W}=\{W_i,w_i\}_{i\in I}$ is called a $K-$ fusion frame controlled by $C$ and $C'$ or $(C,C')-$controlled $K-$ fusion frame if there exist two constantants  \begin{align*}
0<A_{CC'}\leq B_{CC'}<\infty
\end{align*}
such that \begin{equation}\label{def: ckf}
A_{CC'} \Vert K^* f\Vert^2\leq\sum_{i\in I} w_i^2<\pi_{W_i}C f,\pi_{W_i}C' f>\leq B_{CC'} \Vert  f\Vert^2, f \in\mathcal{H}.
\end{equation}
 Where $K^*$ is the adjoint operator of $K$
\end{defi}
$A_{CC'}$ and $B_{CC'}$ are called lower and upper bounds of a $(C,C')-$controlled $K$-fusion frame respectively.
\begin{enumerate}
\item We call $\mathcal{W}$ a $(C,C')$-controlled Parsval  $K-$ fusion frame if \\$A_{CC'}=B_{CC'}=1$.
\item If only the second inequality of \ref{def: ckf} is required, we call $\mathcal{W}$ a $(C,C')$-controlled  Bessel $K-$fusion sequence with Bessel bound $B$.
\end{enumerate}
\begin{rem}\begin{itemize}
\item[i)]
If $K=I$ ( where is the identity operator), then every $(C,C')$-controlled $K-$ fusion frame is a $(C,C')$-controlled fusion frame.
\item[ii)] If $C=C'=I$, then  every $(C,C')$-controlled $K-$ fusion frame is a $K$-fusion frame.
\item[iii)] Every  $(C,C')$-controlled fusion frame is a $(C,C')$-controlled $K$-fusion frame. Indeed,  by definition \ref{def: ckf} there exist constants  $0<A_{CC'}\leq B_{CC'} $, such that for all $f \in \mathcal{H}$ , we have
\begin{equation*}
A_{CC'} \Vert  f\Vert^2\leq\sum_{i\in I} w_i^2<\pi_{W_i}C f,\pi_{W_i}C' f>\leq B_{CC'} \Vert  f\Vert^2.
\end{equation*}
Therefore, for $\Vert K\Vert>0$, one has \begin{align*}
A_{CC'} \Vert K^* f\Vert^2 \leq A_{CC'} \Vert K^*\Vert^2\Vert  f\Vert^2 \leq A_{CC'} \Vert K\Vert^2\Vert  f\Vert^2,
\end{align*}
that is, \begin{align*}
\frac{A_{CC'}}{\Vert K\Vert^2} \Vert K^* f\Vert^2 \leq A_{CC'} \Vert  f\Vert^2, 
\end{align*}
it follows that,\begin{equation*}
\frac{A_{CC'}}{\Vert K\Vert^2} \Vert  K^* f\Vert^2\leq\sum_{i\in I} w_i^2<\pi_{W_i}C f,\pi_{W_i}C' f>\leq B_{CC'} \Vert  f\Vert^2.
\end{equation*}
Hence, the family $\mathcal{W}$ is a  $(C,C')$-controlled $K-$ fusion frame for $\h$.
\end{itemize}
\end{rem}

The next example shows that in general, frames may be controlled $K$-fusion frame  without being a controlled fusion frame. 
\begin{exa} Let $\h=l_2(\mathbb{C})=\{\{a_n\}_{n\in \mathbb{N}} \subset \mathbb{C}\mid \sum_{n=0}^{+\infty}\mid a_n\mid^2<\infty\}$ be a Hilbert space, with respect to the inner product \begin{align*}
\langle \{a_n\}_{n\in \mathbb{N}},\{b_n\}_{n\in \mathbb{N}}\rangle=\sum_{n\in \mathbb{N}}a_n\overline{b}_n,
\end{align*}
equipped with the norm \begin{align*}
\Vert \{a_n\}_{n\in \mathbb{N}}\Vert_{l_2(\mathbb{C})}=(\sum_{n\in \mathbb{N}}\vert a_n\vert^2)^{\frac{1}{2}}.
\end{align*}
Consider two  operators $C$ and  $C'$  defined by \begin{eqnarray*}
C:\quad \h &\longrightarrow &\h\\
\{a_n\}_{n\in \mathbb{N}} &\longmapsto & \{\alpha a_n\}_{n\in \mathbb{N}}
\end{eqnarray*}
 \begin{eqnarray*}
resp.\quad C':\quad \h &\longrightarrow &\h\\
\{a_n\}_{n\in \mathbb{N}} &\longmapsto & \{\beta a_n\}_{n\in \mathbb{N}}
\end{eqnarray*}
where $\alpha,\beta \in \mathbb{R}^*_+$.\\
Tt is easy to see that:
\begin{itemize}
\item $C$ and  $C'$ are positives.
\item $C$ and  $C'$ are invertibles. 
\end{itemize}
There  invertible operators are given respectivelly by:\begin{eqnarray*}
C^{-1}:\quad \h &\longrightarrow &\h\\
\{a_n\}_{n\in \mathbb{N}} &\longmapsto & \{\alpha^{-1} a_n\}_{n\in \mathbb{N}},
\end{eqnarray*}
and
 \begin{eqnarray*}
\quad C'^{-1}:\quad \h &\longrightarrow &\h\\
\{a_n\}_{n\in \mathbb{N}} &\longmapsto & \{\beta^{-1} a_n\}_{n\in \mathbb{N}}.
\end{eqnarray*}
Let $E_i=\{a_j\}_{j\in \mathbb{N}}$, where $a_j=\{\delta_i^j\}_{i\in \mathbb{N}}$(where $\delta_i^j$ is the Kronecker symbol). Let  $\{W_i\}_{i\in \mathbb{N}}$ be a closed subspaces  of  $\mathcal{H}$ such that $W_i=\mathbb{C} E_i$, and let $w_i=\frac{1}{\sqrt{i+1}}$,  for all $i\in\mathbb{N}$.\\
The family $\mathcal{W}=\{W_i,w_i\}_{i\in \mathbb{N}}$ is a $(C,C')$-controlled  Bessel fusion sequence. Indeed  for each  $\{a_n\}_{n\in \mathbb{N}}\in\h$, we have\begin{eqnarray*}
\sum_{i\in \mathbb{N}} w_i^2\langle \pi_{W_i}C (\{a_n\}_{n\in \mathbb{N}}),\pi_{W_i}C'(\{a_n\}_{n\in \mathbb{N}})\rangle &=&\alpha\beta\sum_{i\in \mathbb{N}}\frac{1}{i+1}\mid a_i\mid^2 \\ &\leq & \alpha\beta\sum_{i\in \mathbb{N}}\mid a_i\mid^2\\
&=& \alpha\beta \Vert\{a_n\}_{n\in \mathbb{N}})\Vert^2_{\h}.
\end{eqnarray*}
But is not $(C,C')$-controlled fusion frame, For this, assume the contrary that exists $A_{CC'}>0$ such that: 
\begin{align}
A_{CC'} \sum_{i\in \mathbb{N}} \mid a_i\mid^2 \leq \sum_{i\in \mathbb{N}}\frac{\alpha\beta}{i+1}\mid a_i\mid^2.
\end{align} Hence
\begin{align*}
\sum_{i\in \mathbb{N}} \mid a_i\mid^2< \infty\Longrightarrow \lim_{i\longrightarrow +\infty} a_i=0
\end{align*}
So, we have 
\begin{eqnarray*}
\mid a_j\mid &\longrightarrow& 0 \quad as\quad j\longrightarrow\infty; \\ 
\frac{\alpha\beta}{i+1}&\longrightarrow& 0 \quad  as \quad i\longrightarrow\infty.
\end{eqnarray*}
$$\Longrightarrow
\left\{
\begin{array}{rl}  \forall \varepsilon \geq 0\quad \exists N\in \mathbb{N}:\quad j\geq N & \mbox{ $\Longrightarrow $} \mid a_j\mid <\varepsilon,\\
\forall \gamma \geq 0\quad \exists N\in \mathbb{N}:\quad j\geq M  &\mbox{$\Longrightarrow$} \frac{\alpha\beta}{j+1} <\gamma .
\end{array}
\right.
$$

\begin{align*}
 \Longrightarrow \quad \sum_{i\in \mathbb{N}}\frac{\alpha\beta -(j+1)A_{CC'}}{j+1}\mid a_j\mid^2 \geq 0.
\end{align*}
By fixing $ \varepsilon =\gamma$, there exist $N,M\in\mathbb{N}^*$, such that
$$
\left\{
\begin{array}{rl}  j\geq N & \mbox{ $\Longrightarrow $}  \mid a_j\mid <\varepsilon,\\
 j\geq M & \mbox{ $\Longrightarrow $} \frac{\alpha\beta}{j+1} <\varepsilon .
\end{array}
\right.
$$

Now, let $N_1=\max(N,M)$, then $\forall j\geq N_1$, $\mid a_j\mid <\varepsilon \quad and \quad \frac{\alpha\beta}{i+1} <\varepsilon$. 
Hence
\begin{align*}
\sum_{i=0}^{N_1-1} \frac{\alpha\beta -(i+1)A_{CC'}}{i+1}\mid a_i\mid^2+\varepsilon^2 \sum_{i=N_1}^{\infty} (\varepsilon -A_{CC'})\geq  0.
\end{align*}
Now,  for $\varepsilon=\frac{A_{CC'}}{2}$, we obtain
  \begin{align*}
\sum_{i=0}^{N_1-1} \frac{\alpha\beta -(i+1)A_{CC'}}{i+1}\mid a_i\mid^2+(\frac{A_{CC'}}{2})^2 \sum_{i=N_1}^{\infty} (\frac{-A_{CC'}}{2})\geq  0.
\end{align*}
absurde.
\\
Now if we considere  the operator  \begin{eqnarray*}
\quad K:\quad \h &\longrightarrow &\h\\
\{a_n\}_{n\in \mathbb{N}} &\longmapsto & \{\frac{a_n}{\sqrt{n+1}} \}_{n\in \mathbb{N}}.
\end{eqnarray*}
then, $K$ is a bounded linear  for $\mathcal{H}$. furthermore, for each  $\{a_n\}_{n\in \mathbb{N}}\in\h$, we have
\begin{align*}
\langle K^* (\{a_n\}_{n\in \mathbb{N}}),K^*(\{a_n\}_{n\in \mathbb{N}})\rangle = \sum_{i=0}^{\infty}\frac{\mid a_n \mid^2}{n+1}.
\end{align*}
So,
\begin{eqnarray*}
\sum_{i=0}^{\infty} \frac{\alpha\beta}{2}\frac{\mid a_n \mid^2}{n+1} \leq \sum_{i=0}^{\infty} \frac{\alpha\beta}{n+1}\mid a_n \mid^2\leq \alpha\beta  \sum_{i=0}^{\infty} \mid a_n \mid^2, 
\end{eqnarray*}
\end{exa}
The following proposition provides a relation between controlled $K$-fusion frames and controlled fusion frames.
\begin{prop}
Let $K\in\mathcal{B}(\h)$ be a closed range operator  $\mathcal{R}_K$. Then, every $(C,C')$-controlled $K-$ fusion frame is a $( C,C')$-controlled fusion frame for $\mathcal{R}_K$.

\end{prop}
\begin{proof}
 Let $\mathcal{W}=\{(W_i,w_i)\}_{i\in I}$ be a $(C,C')$-controlled $K-$ fusion frame with frame bounds $A_{CC'}$ and $B_{CC'}$. Then for all $f \in \mathcal{R}_K $, we have
\begin{equation*}
A_{CC'} \Vert  K^* f\Vert^2\leq\sum_{i\in I} w_i^2<\pi_{W_i}C f,\pi_{W_i}C' f>\leq B_{CC'} \Vert  f\Vert^2.
\end{equation*}
Therefore, via lemma \ref{lem: Pseudo-inverse}, we have
\begin{align*}
A_{CC'} \Vert   f\Vert^2\leq  A_{CC'}\Vert  (K^*)^\dagger f\Vert^2 \Vert  K^* f\Vert^2.
\end{align*}
Hence,\begin{align*}
\frac{A_{CC'}}{\Vert  (K^*)^\dagger \Vert^2} \Vert   f\Vert^2\leq A_{CC'}\Vert  K^* f\Vert^2.
\end{align*}
Thus, \begin{equation*}
\frac{A_{CC'}}{\Vert  (K^*)^\dagger \Vert^2} \Vert   f\Vert^2 \leq\sum_{i\in I} w_i^2<\pi_{W_i}C f,\pi_{W_i}C' f>\leq B_{CC'} \Vert  f\Vert^2.
\end{equation*}
So, we have the result.
\end{proof}
If $\mathcal{W}$ is a $(C,C')$-controlled $K-$ fusion frame and $C'^*\pi_{W_i}C$ is a positive operator for each $i\in I$, then $C'^*\pi_{W_i}C=C^*\pi_{W_i} C'$ and we have \begin{align*}
 A_{CC'} \Vert K^* f\Vert^2\leq\sum_{i\in I} w_i^2\Vert (C'^*\pi_{W_i}C)^{\frac{1}{2} }f\Vert^2\leq B_{CC'} \Vert  f\Vert^2,  f \in \mathcal{H}.
\end{align*}
 Indeed, 
  \begin{eqnarray*}
  \sum_{i\in I} w_i^2 \langle \pi_{W_i}C f,\pi_{W_i}C' f\rangle &=& \sum_{i\in I} w_i^2  \langle C'^*\pi_{W_i}C f, f\rangle\\
  &=& \sum_{i\in I} w_i^2 \langle (C'^*\pi_{W_i}C)^{\frac{1}{2}}f, (C'^*\pi_{W_i}C)^{\frac{1}{2}}f \rangle\\
  &=& \sum_{i\in I} w_i^2 \Vert (C'^*\pi_{W_i}C)^{\frac{1}{2} }f\Vert^2.
  \end{eqnarray*}
 We define the controlled analysis operator by 
 \begin{eqnarray*}
  T_{CC'}:\quad \mathcal{H} &\longrightarrow & \mathcal{K}\\ f &\longmapsto &T_{CC'}(f):=(w_i (C'^*\pi_{W_i}C)^{\frac{1}{2}}f)_{i\in I},
\end{eqnarray*}
 Where\begin{align*}
 \mathcal{K}=\{(w_i (C'^*\pi_{W_i}C)^{\frac{1}{2}}f)_{i\in I}  \vert f\in \mathcal{H}\}\subseteq (\oplus_{i\in I}\mathcal{H})_{l^2}.
 \end{align*}
 $\mathcal{K}$ is closed \cite{KM12} and $ T_{CC'}$ is well defined. Morever $T_{CC'}$ is a bounded linear operator. Its adjoint operator is given by   \begin{eqnarray*}
  T^*_{CC'}:\quad \mathcal{K} &\longrightarrow & \mathcal{H}\\(w_i (C'^*\pi_{W_i}C)^{\frac{1}{2}}f)_{i\in I}) &\longmapsto &T^*_{CC'}((w_i (C'^*\pi_{W_i}C)^{\frac{1}{2}}f)_{i\in I}):=\sum_{i\in I} w_i^2C'^*\pi_{W_i}C f,
\end{eqnarray*}
and is called the controlled synthesis operator.\\
Therefore, we define the controlled $K$-fusion frame operator $S_{CC'}$ on $\mathcal{H}$ by\begin{align}\label{def:opers}
S_{CC'}=T^*_{CC'}T_{CC'}(f)=\sum_{i\in I} w_i^2 C'^*\pi_{W_i}C f, \quad  f\in\mathcal{H}.
\end{align}

In fact, many of the properties of the ordinary $K$-fusion frames are valid in this case.
\begin{lem}
 Let  $\mathcal{W}=\{W_i,w_i\}_{i\in I}$ be a $(C,C')$-controlled $K$-fusion frame with bounds $A_{CC'}$ and $B_{CC'}$. Then the operator  $S_{CC'}$  (\ref{def:opers})  is a well defined, linear, positive, bounded and self-adjoint operator. furthermore, we have \begin{align}
 A_{CC'} K K^* \leq S_{CC'}\leq B_{CC'}Id_{\mathcal{H}}.
 \end{align}
\end{lem}
\begin{proof}
\begin{itemize}
\item By definition, $S_{CC'}$ is  a linear bounded and  well defined operator, and it is clear to see that $S_{CC'}$ is a positive and self-adjoint operator. 
\item The family $\mathcal{W}=\{W_i,w\}_{i\in I}$ is a $(C,C')$-controlled $K-$fusion frame for $\mathcal{H}$ with bounds $A_{CC'}$ and $B_{CC'}$ if and only if
 $$ A_{CC'}\Vert K^* f\Vert^2\leq \langle  S_{CC'}f,f\rangle= \langle\sum_{i\in I} w_i^2 C'^*\pi_{W_i}C f,f\rangle \leq B_{CC'} \Vert f\Vert^2, \quad  f \in \mathcal{H},$$ 
that is,\begin{align*}
A_{CC'}\langle KK^* f,f\rangle\leq \langle S_{CC'}f,f\rangle \leq B_{CC'} \langle f,f\rangle, \quad  f \in \mathcal{H}.
\end{align*}
Hence,\begin{align*}
A_{CC'} KK^* \leq  S_{CC'} \leq  B_{CC'}.Id_{\mathcal{H}}, 
\end{align*}
so the conclusion holds. 
\end{itemize}
\end{proof}

The next theorem generalizes  the situation of controlled  Bessel $K-$fusion sequence. Since it has similar procedure, the proof is omitted.
  \begin{thm}\label{thm: operatobiendefini}
 $\mathcal{W}$ is a $(C,C')$-controlled  Bessel $K-$fusion sequence with bound $B_{CC'}$ if and only if \begin{eqnarray*}
  T^*_{CC'}:\quad \mathcal{K} &\longrightarrow & \mathcal{H}\\ (w_i (C'^*\pi_{W_i}C)^{\frac{1}{2}}f)_{i\in I}&\longmapsto &T^*_{CC'}((w_i (C'^*\pi_{W_i}C)^{\frac{1}{2}}f)_{i\in I}):=\sum_{i\in I} w_i^2C'^*\pi_{W_i}C f,
\end{eqnarray*} 
is well-defined bounded operator and $\Vert T^*_{CC'}\Vert\leq\sqrt{B}$.
 \end{thm}
  Controlled K-fusion frame operator of $(C,C')$-controlled $K$-fusion frame is not invertible in general, but we can show that it is invertible on the subspace $\mathcal{R}_K \subset \mathcal{H}$. In fact, since $\mathcal{R}_K$ is closed \begin{align*}
 KK^\dagger\mid_{\mathcal{R}_K}= id_{\mathcal{R}_K},
 \end{align*}
so we have \begin{align*}
 id^*_{\mathcal{R}_K}=(K^\dagger\mid_{\mathcal{R}_K})^*K^*.
\end{align*}
Hence for any $f\in\mathcal{R}_K $
\begin{eqnarray*}
\Vert f\Vert =\Vert (K^\dagger\mid_{\mathcal{R}_K})^*K^* f\Vert\leq \Vert K^\dagger\Vert \Vert K^* f\Vert,
\end{eqnarray*}
that is, $\Vert K^* f\Vert^2\geq \Vert K^\dagger\Vert^{-2}\Vert \Vert f \Vert^2$. Combined whith \ref{def: ckf} we have \begin{align*}
\langle S_{CC'}f,f\rangle \geq A_{CC'}\Vert K^* f\Vert^2\geq A_{CC'}\Vert K^\dagger\Vert^{2}\Vert f\Vert^2,\quad \forall f\in\mathcal{R}_K. 
\end{align*}
So from the definition of $(C,C')$-controlled $K$-fusion frame, which implies that $S:\quad \mathcal{R}_K\longrightarrow S(\mathcal{R}_K)$ is an isomorphism, furthermore we have\begin{align*}
B^{-1}_{CC'}\Vert f\Vert \leq \Vert S^{-1} f\Vert\leq A^{-1}_{CC'}\Vert K^\dagger\Vert^{2}\Vert f\Vert,\forall f\in (S(\mathcal{R}_K)).
\end{align*}

\begin{thm}
 Let $K\in B(\mathcal{H})$ be a closed range operator $\mathcal{R}_K$, $\mathcal{W}$ is a $(C,C)$-controlled  $K-$fusion frame with bounds $A_{CC'}$ and $B_{CC'}$ if and only if \begin{eqnarray*}
  T^*_{CC'}:\quad\quad \mathcal{K} &\longrightarrow & \mathcal{H}\\ (w_i (C'^*\pi_{W_i}C)^{\frac{1}{2}}f)_{i\in I}&\longmapsto &T^*_{CC'}((w_i (C'^*\pi_{W_i}C)^{\frac{1}{2}}f)_{i\in I}):=\sum_{i\in I} w_i^2C'^*\pi_{W_i}C f,
\end{eqnarray*} 
is well-defined and surjective
\end{thm}
\begin{proof}
Let the sequence  $\mathcal{W}$ be a $(C,C')$-controlled  $K-$fusion frame for $\h$, and let $S_{CC'}$ be its controlled  $K-$fusion frame operator. Then, it is a $(C,C')$-controlled  Bessel $K-$fusion sequence  and therefore, by Theorem \ref{thm: operatobiendefini} , the bounded operator $T^*_{CC'}$ is well-defined.
It remains to show that  
By definition, for each $f\in\h$, we have 
\begin{equation*}
A_{CC'} \Vert  K^* f\Vert^2\leq\sum_{i\in I} w_i^2<\pi_{W_i}C f,\pi_{W_i}C' f>\leq B_{CC'} \Vert  f\Vert^2.
\end{equation*}
In particular, we have 
\begin{equation*}
A_{CC'} \Vert  K^* f\Vert^2\leq \langle S_{CC'}f, f\rangle\leq \Vert S_{CC'}\Vert \Vert  f\Vert.
\end{equation*}
Since, $\quad S_{CC'}= T^*T$, then\begin{align*}
\Vert S_{CC'}\Vert \Vert  f\Vert \leq \Vert T \Vert \Vert Tf\Vert \Vert f\Vert.
\end{align*}
Hence,\begin{align*}
A_{CC'}\Vert T \Vert^{-1} \Vert  K^* f\Vert^2\leq \Vert Tf\Vert \Vert f\Vert.
\end{align*}
Therefore, via lemma \ref{lem: Pseudo-inverse}, we have
\begin{align*}
A_{CC'} \Vert   f\Vert^2\leq  A_{CC'}\Vert  (K^*)^\dagger \Vert^2 \Vert  K^* f\Vert^2.
\end{align*}
Hence, \begin{equation*}
\frac{A_{CC'}}{\Vert  (K^*)^\dagger \Vert^2}  \Vert  f\Vert^2 \leq A_{CC'}\Vert  K^* f\Vert^2.
\end{equation*}
\begin{align*}
\frac{A_{CC'}}{\Vert T \Vert \Vert  (K^*)^\dagger \Vert^2} \Vert   f\Vert \leq \Vert Tf\Vert  
\end{align*}
Thus, $T^*_{CC'}$ is surjective.\\
Conversely, let $T^*_{CC'}$ be a well-defined, bounded and surjective, then theorem \ref{thm: operatobiendefini} shows that  $ \mathcal{W}$ is a $(C,C')$-controlled  Bessel $K$-fusion sequence for $\h$.  Therefore, for each $f\in\h$, since $T^*_{CC'}$ is surjective, then, by Lemma \ref{lem: Pseudo-inverse}, there exists an operator  $(T_{CC'}^*)^\dagger :\quad \h \longrightarrow \mathcal{K}$, such that   \begin{align*}
 T_{CC'}^*(T_{CC'}^*)^\dagger= id. 
 \end{align*}
 Hence, \begin{align*}
 T_{CC'}^\dagger T_{CC'} = id. 
 \end{align*}
So, for each $f\in\h$, we have
 \begin{eqnarray*}
 \Vert   K^* f\Vert^2 &\leq & \Vert K \Vert^2  \Vert T_{CC'}^\dagger \Vert^2 \Vert T_{CC'} \Vert^2\\
&=& \Vert T_{CC'}^\dagger \Vert^2  \Vert K \Vert^2  \sum_{i\in I} w_i^2<\pi_{W_i}C f,\pi_{W_i}C' f>.
\end{eqnarray*}
Therefore, $\mathcal{W}$ is a $(C,C')$-controlled  $K-$fusion frame for $\h$.
\end{proof}
\begin{prop}
Let $K\in B(\mathcal{H})$, $C,C'\in GL^+(\h)$ and let $\mathcal{W}$ be a $(C,C')$-controlled  $K-$fusion frame for $\h$ with bounds $A_{CC'}$ and $B_{CC'}$ with $\overline{\mathcal{R}}_{T^*}$ is orthogonally complemented. If $T\in B(\h)$ with $\mathcal{R}_{T}\subset \mathcal{R}_{K}$. Then $\mathcal{W}$ is a $(C,C')$-controlled  $T$-fusion frame for $\h$.
\end{prop}
\begin{proof}
Assume that $\mathcal{W}$ be a $(C,C')$-controlled  $K-$fusion frame for $\h$ with bounds $A_{CC'}$ and $B_{CC'}$. Then for each $f\in\h$, we have \begin{align*}
A_{CC'} \langle K^* f, K^* f \rangle\leq\sum_{i\in I} w_i^2\langle \pi_{W_i}C f,\pi_{W_i}C f\rangle\leq B_{CC'} \langle  f,f\rangle.
\end{align*} 
Since  $\mathcal{R}_{T}\subset \mathcal{R}_{K}$, so by using lemma \ref{lem: orthogonally complemented} , there exists some $ \lambda >0$ such that \begin{align*}
 TT^*\leq \lambda K K^*.
\end{align*}
This implies that for all $f\in \h$, we have \begin{align*}
A_{CC'} \langle T^* f, T^* f \rangle\leq A_{CC'} \lambda\langle K^* f, K^* f \rangle.
\end{align*}
Therfore, \begin{align*}
\frac{A_{CC'}}{\lambda} \langle T^* f, T^* f \rangle\leq  A_{CC'} \lambda\langle K^* f, K^* f \rangle  \leq\sum_{i\in I} w_i^2\langle \pi_{W_i}C f,\pi_{W_i}C f\rangle\leq B_{CC'} \langle  f,f\rangle.
\end{align*}
Then, $\mathcal{W}$ is a $(C,C')$-controlled  $T$-fusion frame for $\h$ with bounds $\frac{A_{CC'}}{\lambda}$ and $B_{CC'}$.
\end{proof}
\begin{thm}
Let $K_1, K_2\in B(\mathcal{H})$ such that $\mathcal{R}_{K_1}\perp \mathcal{R}_{K_2}$. If $\mathcal{W}$  is a $(C,C')$-controlled  $K_i$-fusion frame for $\h$ $(i=1,2)$. Then $\mathcal{W}$ is a $(C,C')$-controlled  $(\alpha K_1+\beta K_1)$-fusion frame for $\h$, where $\alpha,\beta\in \mathbb{C}$.
\end{thm}
\begin{proof}
Since $\mathcal{W}$  is a $(C,C')$-controlled  $K_i$-fusion frame for $\h$ $(i=1,2)$, there exist $A^j_{CC'},B^j_{CC'}>0$, such that for all $f\in \h$, $j=1,2$, we have \begin{align*}
A^j_{CC'}\langle K_j^* f, K_j^* f \rangle\leq \sum_{i\in I} w_i^2\langle \pi_{W_i}C f,\pi_{W_i}C f\rangle\leq B^j_{CC'} \langle  f,f\rangle.
\end{align*}
Then for any $f\in \h$, we have 
\begin{align*}
\langle (\alpha K_1+\beta K_2)^*f,(\alpha K_1+\beta K_2)^*f\rangle = \langle \overline{\alpha }K^*_1 f+ \overline{\beta} K^*_2 f,\overline{\alpha }K^*_1 f+ \overline{\beta} K^*_2 f\rangle 
\end{align*}
\begin{align*}
=\vert \alpha\vert^2 \langle K^*_1 f,K^*_1 f \rangle+ \overline{\alpha}\beta \langle K^*_1 f,K^*_2 f \rangle +\alpha\overline{\beta} \langle K^*_2 f,K^*_1 f \rangle+\vert \beta\vert^2 \langle K^*_2 f,K^*_2 f \rangle.
\end{align*}
Since $\mathcal{R}_{K_1}\perp \mathcal{R}_{K_2}$, then, for any $f\in \h$, we have
\begin{eqnarray*}
\langle K^*_1 f,K^*_2 f \rangle &=& 0,\\
\langle K^*_2 f,K^*_1 f \rangle &=& 0.\\
\end{eqnarray*}
Thus, \begin{align*}
\langle (\alpha K_1+\beta K_2)^*f,(\alpha K_1+\beta K_2)^*f\rangle = \vert \alpha\vert^2 \langle K^*_1 f,K^*_1 f \rangle+ \vert \beta\vert^2 \langle K^*_2 f,K^*_2 f \rangle.
\end{align*}
Therfore, for any $f\in \h$, we have 
\begin{align*}
\frac{A^1_{CC'} A^2_{CC'}}{2( \vert \alpha\vert^2 A^1_{CC'}+ \vert \beta\vert^2 A^2_{CC'})}\langle (\alpha K_1+\beta K_2)^*f,(\alpha K_1+\beta K_2)^*f\rangle  
\end{align*}\begin{eqnarray*}
&=&\frac{A^1_{CC'} A^2_{CC'}\vert \alpha\vert^2}{2( \vert \alpha\vert^2 A^1_{CC'}+ \vert \beta\vert^2 A^2_{CC'})}\langle K^*_1 f,K^*_1 f \rangle+\frac{A^1_{CC'} A^2_{CC'}\vert \beta\vert^2}{2( \vert \alpha\vert^2 A^1_{CC'}+ \vert \beta\vert^2 A^2_{CC'})}\langle K^*_2 f,K^*_2 f \rangle \\
&\leq & \frac{1}{2}( \sum_{i\in I} w_i^2\langle \pi_{W_i}C f,\pi_{W_i}C f\rangle  + \sum_{i\in I} w_i^2\langle \pi_{W_i}C f,\pi_{W_i}C f\rangle)\\
&\leq &\frac{B^1_{CC'} +B^2_{CC'} }{2}\langle  f,f\rangle.
\end{eqnarray*}
Thus, $\mathcal{W}$ is a $(C,C')$-controlled  $\alpha K_1+\beta K_2$-fusion frame with bounds $\frac{A^1_{CC'} A^2_{CC'}}{2( \vert \alpha\vert^2 A^1_{CC'}+ \vert \beta\vert^2 A^2_{CC'})}$ and $\frac{B^1_{CC'} +B^2_{CC'} }{2}$.
\end{proof}
\begin{lem}
Let $K\in B(\mathcal{H})$, and  $C,C'\in GL^+(\h)$. Assume that \\$CK=KC$, $C'K=KC'$ and $SC=CS$. Then, $\mathcal{W}$ is a $(C,C')$-controlled  $K$-fusion frame for $\h$  if and only if $\mathcal{W}$ is a $K$-fusion frame for $\h$.  \\
Where $S$ is the $K$-fusion frame operator (\ref{sfusionframeoper}), defined by $$ Sf=\sum_{i\in I} w_i^2\ \pi_{W_i} f,\quad f\in\h .$$, .
\end{lem}
\begin{proof}
Assume that  $\mathcal{W}$ is a $K-$fusion frame with bounds $A$ and $B$. Then for each $f\in\h$, we have
\begin{align*}
A \Vert K^* f\Vert^2\leq\sum_{i\in I} w_i^2\Vert \pi_{W_i} f\Vert^2\leq B  \Vert  f\Vert^2.
\end{align*}
Since, $C$ and $C'$ are linear bounded operators, applying \ref{prop: mmMM}, there exist constants $m,m'$, $M$ and $M'>0$ such that
$$
\left\{
\begin{array}{rl}  m I \leq & \mbox{ C }\leq M I,\\
m' I\leq &\mbox{C'}\leq M'I .
\end{array}
\right.
$$
\begin{align*}
\langle SCf,f \rangle= \langle f, CSf\rangle.
\end{align*}
Then, \begin{align*}
m K K^*\leq CS\leq MS\leq MBI.
\end{align*}
We deduce that \begin{align*}
m m'K K^*\leq C'SC\leq M M'BI.
\end{align*}
Therfore, for each $f\in\h$, we have\begin{align*}
m m' A\langle K^*f,K^*f \rangle\leq \sum_{i\in I} w_i^2\langle \pi_{W_i}C f\pi_{W_i}C' f\rangle\leq  M M'B \Vert  f\Vert^2. 
\end{align*}
Thus, $\mathcal{W}$ is a $(C,C)$-controlled  $K-$fusion frame.\\
Conversely, Assume that  $\mathcal{W}$ is a $(C,C)$-controlled  $K-$fusion frame with bounds $A$ and $B$. Then for each $f\in\h$, we have
\begin{align*}
A_{CC'} \Vert K^* f\Vert^2\leq\sum_{i\in I} w_i^2\langle \pi_{W_i} Cf,\pi_{W_i} C'f\leq B _{CC'} \Vert  f\Vert^2.
\end{align*}
For each $f\in \h$, we have
\begin{eqnarray*}
A_{CC'} \langle K^*f,K^*f\rangle &=& A _{CC'}\langle (CC')^{-\frac{1}{2}}(CC')^{\frac{1}{2}} K^*f, (CC')^{-\frac{1}{2}}(CC')^{\frac{1}{2}}K^*f\rangle \\
&\leq & A _{CC'}\Vert(CC')^{\frac{1}{2}}\Vert^2\sum_{i\in I} w_i^2 \langle  \pi_{W_i}C(CC')^{-\frac{1}{2}} f, \pi_{W_i}C (CC')^{-\frac{1}{2}}f\rangle\\
&=& A _{CC'} \Vert(CC')^{\frac{1}{2}}\Vert^2\sum_{i\in I} w_i^2 \langle  \pi_{W_i}(C)^{\frac{1}{2}} (C')^{-\frac{1}{2}} f, \pi_{W_i}(C)^{\frac{1}{2}} (C')^{-\frac{1}{2}}f\rangle\\
&=& A _{CC'}\Vert(CC')^{\frac{1}{2}}\Vert^2 \langle \sum_{i\in I} w_i^2 \pi_{W_i}(C)^{\frac{1}{2}} (C')^{-\frac{1}{2}} f, \pi_{W_i}(C)^{\frac{1}{2}} (C')^{-\frac{1}{2}}f\rangle\\
&=& A _{CC'}\Vert(CC')^{\frac{1}{2}}\Vert^2 \langle \sum_{i\in I} w_i^2 \pi_{W_i} f, f\rangle.\\
\end{eqnarray*}
\begin{align*}
\Longrightarrow \quad A_{CC'}\Vert(CC')^{\frac{1}{2}}\Vert^{-2} \langle K^*f,K^*f\rangle \leq  \sum_{i\in I} w_i^2 \langle  \pi_{W_i}f,\pi_{W_i}f\rangle.
\end{align*}
In the other hand \begin{align*}
 \sum_{i\in I} w_i^2 \langle  \pi_{W_i}f,\pi_{W_i}f\rangle =\langle Sf,f\rangle,
\end{align*}
where  $\quad\quad Sf=\sum_{i\in I} w_i^2  \pi_{W_i}f$.
\begin{eqnarray*}
\langle Sf,f\rangle &=& \langle (CC')^{-\frac{1}{2}} (CC')^{\frac{1}{2}} Sf,f\rangle\\
&=&\langle (CC')^{\frac{1}{2}} Sf,(CC')^{-\frac{1}{2}} f\rangle\\
&=&\langle (CC')^{\frac{1}{2}} Sf,(CC')^{-\frac{1}{2}} f\rangle\\&=&\langle C'SC(CC')^{-\frac{1}{2}} f,(CC')^{-\frac{1}{2}} f\rangle\\
&\leq& B_{CC'} \Vert(CC')^{-\frac{1}{2}}\Vert^2 \Vert f\Vert^2.
\end{eqnarray*}
Thus,  $\mathcal{W}$ is a $K-$fusion frame  with bounds $A_{CC'}\Vert(CC')^{\frac{1}{2}}\Vert^{-2} $ and $B_{CC'} \Vert(CC')^{-\frac{1}{2}}\Vert^2$.\\
\end{proof}
  \begin{thm}
Let $K\in B(\mathcal{H})$, let $\mathcal{W}$ be a $(C,C)$-controlled  $K$-fusion frame with bounds $A_{CC}$ and $B_{CC}$. If $U\in \mathcal{B(\mathcal{H}})$ is an invertible operator such that $U^*C = CU^*$ and $K^*(U^*)^{-1}=(U^*)^{-1}K^*$, then ${(U W_i, w_i)}_{i\in I} $ is a $(C,C)$-controlled $K$-fusion frame for $\h$.
 \end{thm}
\begin{proof}
Assume that $\mathcal{W}$ is a $(C,C)$-controlled  $K-$fusion frame with bounds $A_{CC}$ and $B_{CC}$.\\
By definition, for each $f\in\h$, we have \begin{align*}
A_{CC} \Vert K^* f\Vert^2\leq\sum_{i\in I} w_i^2\Vert \pi_{W_i}C f\Vert^2\leq B_{CC} \Vert  f\Vert^2.
\end{align*}
Now, let $f\in\h$. Via lemma \ref{lem: Projection-operator} and since $UW_i$ is closed, we have
\begin{eqnarray*}
\Vert \pi_{W_i}C U^*f\Vert=\Vert \pi_{W_i} U^* Cf\Vert =\Vert \pi_{W_i}U^*\pi_{\overline{UW_i}} Cf\Vert &=&\Vert \pi_{W_i}U^*\pi_{UW_i} Cf\Vert\\ &\leq & \Vert U\Vert \Vert \pi_{UW_i} Cf\Vert.
\end{eqnarray*}
Therefore, 
\begin{eqnarray*}
A_{CC} \Vert K^* U^* f\Vert^2  &\leq &   \sum_{i\in I} w_i^2\Vert \pi_{W_i}C U^* f\Vert^2 \\ &\leq &  \Vert U\Vert^2 \sum_{i\in I} w_i^2\Vert \pi_{UW_i}C f\Vert^2.
\end{eqnarray*}
\begin{eqnarray*}
A_{CC}\Vert K^*  f\Vert^2 =\Vert K^* (U^*)^{-1} U^* f\Vert^2 &=& \Vert (U^*)^{-1} K^*  U^* f\Vert^2 \\ &\leq& \Vert U^{-1}\Vert^2 \Vert K^* U^* f\Vert^2.
\end{eqnarray*}
Then, we have\begin{align}\label{eq: ccontrolled1}
\frac{A}{\Vert U^{-1}\Vert^{-2} \Vert U\Vert^{-2}}\Vert K^*  f\Vert^2 \leq \sum_{i\in I} w_i^2\Vert \pi_{UW_i}C f\Vert^2. 
\end{align}
On the other hand,  Via lemma \ref{lem: Projection-operator}, we obtain with $U^{-1}$ instead of $T$:\begin{align*}
\pi_{UW_i}=\pi_{UW_i} (U^*)^{-1} \pi_{W_i}U^*.
\end{align*}
Thus,\begin{eqnarray*}
\Vert \pi_{UW_i}Cf\Vert^2 &=& \Vert \pi_{UW_i} (U^*)^{-1} \pi_{W_i}U^*Cf\Vert^2\\&\leq& \Vert U^{-1}\Vert^{2} \Vert \pi_{W_i} U^*Cf\Vert^2,
\end{eqnarray*}
and it follows \begin{align*}
\sum_{i\in I} w_i^2 \Vert \pi_{UW_i}Cf\Vert^2  \leq  \Vert U^{-1}\Vert^{2} \sum_{i\in I} w_i^2\Vert \pi_{W_i} U^*Cf\Vert^2.
\end{align*}
hence, \begin{align*}
  \sum_{i\in I} w_i^2\Vert \pi_{UW_i}Cf\Vert^2 \leq B_{CC} \Vert U^{-1}\Vert^{2}
   \Vert U\Vert^2 \Vert f\Vert^2.
\end{align*} 
Thus, $\mathcal{W}$ is a $(C,C)$-controlled  $K-$fusion frame with bounds $A_{CC}\Vert U^{-1}\Vert^{-2}
   \Vert U\Vert^{-2}$ and $B_{CC} \Vert U^{-1}\Vert^{2}
   \Vert U\Vert^2$.
\end{proof} 
 \begin{cor}Let $K\in B(\mathcal{H})$, let $\mathcal{W}$ be a $(C,C)$-controlled  $K-$fusion frame with bounds $A_{CC}$ and $B_{CC}$
If $U\in \mathcal{B(\mathcal{H}})$ is a  unitary operator such that $U^{-1}C = CU^{-1}$ and $K^* U=U K^*$, then ${(U W_i, w_i)}_{i\in I} $ is a $(C,C)$-controlled $K$-fusion frame for $\h$.
\end{cor}
\begin{proof} Assume that $\mathcal{W}$ is a $(C,C)$-controlled  $K-$fusion frame with bounds $A_{CC}$ and $B_{CC}$.\\
By definition, for each $f\in\h$, we have \begin{align*}
A_{CC} \Vert K^* f\Vert^2\leq\sum_{i\in I} w_i^2\Vert \pi_{W_i}C f\Vert^2\leq B_{CC} \Vert  f\Vert^2.
\end{align*}
Now, let $f\in\h$. Since, 
\begin{align*}
\Vert \pi_{W_i}C U^{-1}f\Vert=\Vert \pi_{W_i}U^{-1}C f\Vert
\end{align*}
Via \ref{lem: Projection-operator}, we obtain
\begin{align*}
\pi_{W_i} U^{-1} =\pi_{W_i} U^{-1} \pi_{UW_i} .
\end{align*}
since by assumption, we have  $(U^*=U^{-1})$.\\
It follows that \begin{eqnarray*}
\Vert \pi_{W_i}C U^{-1}f\Vert &=& \Vert \pi_{W_i} U^{-1} \pi_{UW_i}C f\Vert\\
&\leq &\Vert U^{-1}\Vert \Vert \pi_{UW_i}C f \Vert.
\end{eqnarray*}
Therefore, we have\begin{eqnarray*}
A_{CC}\Vert K^*U^{-1}f\Vert^2 &\leq &\sum_{i\in I} w_i^2\Vert \pi_{W_i} C U^{-1}f\Vert^2\\
&\leq &\Vert U^{-1}\Vert^2 \sum_{i\in I}  w_i^2 \Vert \pi_{UW_i} C f\Vert^2.
\end{eqnarray*}
And since $K^* U=U K^*$, we have \begin{eqnarray*}
\Vert K^* f\Vert^2 &=& \Vert K^* U U^{-1}f\Vert^2\\
&=&\Vert U K^*  U^{-1}f\Vert^2\\
&\leq & \Vert U\Vert^2\Vert K^* U^{-1}f\Vert^2.
\end{eqnarray*}
Thus, \begin{align}\label{eq: uu1}
\frac{A_{CC}}{\Vert U\Vert^2\Vert  U^{-1}\Vert^2}\Vert K^* f\Vert^2\leq\sum_{i\in I}  w_i^2 \Vert \pi_{UW_i} C f\Vert^2.
\end{align}
On the other hand, since $\pi_{UW_i}U=U\pi_{W_i}$, we have \begin{eqnarray*}
\sum_{i\in I}w_i^2 \Vert \pi_{UW_i} C f\Vert^2 &=& \sum_{i\in I}w_i^2 \Vert \pi_{UW_i} U U^{-1}C f\Vert^2\\
&=& \sum_{i\in I}w_i^2 \Vert U \pi_{W_i}  U^{-1}C f\Vert^2\\
&\leq &\Vert U\Vert^2\sum_{i\in I}w_i^2 \Vert \pi_{W_i}  U^{-1}C f\Vert^2\\
&\leq & B_{CC}\Vert U\Vert^2 \Vert U^{-1}\Vert^2 \Vert F\Vert^2.
\end{eqnarray*}
Thus, \begin{align}\label{eq: uu2}
\sum_{i\in I}w_i^2 \Vert \pi_{UW_i} C f\Vert^2\leq  B_{CC}\Vert U\Vert^2 \Vert U^{-1}\Vert^2 \Vert F\Vert^2.
\end{align}
By \ref{eq: uu1} and \ref{eq: uu2} we conclude that $\mathcal{W}$ is a $(C,C)$-controlled  $K-$fusion frame with bounds $A_{CC}\Vert U^{-1}\Vert^{-2}
   \Vert U\Vert^{-2}$ and $B_{CC} \Vert U^{-1}\Vert^{2}
   \Vert U\Vert^2$.
\end{proof} 
 \section{Perturbation on Controlled $K$-fusion frame}
The following result provides a sufficient condition on a family of closed subspaces of $\mathcal{H}$ to be a controlled $K$-fusion frame, in the precence of another controlled $K$-fusion frame. In fact it is a generalisation of Proposition $2.4$ in \cite{AK05}, Proposition $4.6$ in \cite{CH97} and Proposition $2.6$ in \cite{KM12}.
\begin{prop}
Let $ K\in B(\mathcal{H})$ be a closed range operator $\mathcal{R}_{K}$, let $T, U\in GL(\mathcal{H})$ and let $\mathcal{W}=\{W_i,w_i\}_{i\in I}$ be a $(C,C')$-controlled  $K-$fusion frame  for $\mathcal{H}$ with lower and upper  bounds $A_{CC'}$ and $B_{CC'}$, respectively. Let $\{V_i\}_{i\in I}$ be a family of closed subspaces of $\mathcal{H}$. if there exists a number $0< R <A_{CC'}$ such that \begin{align}
0<\sum_{i\in I} w_i^2 \langle C'^*( \pi_{V_i}-\pi_{W_i})C f, f\rangle \leq R\Vert f\Vert^2, \forall f\in \mathcal{H},
\end{align}
then $\mathcal{V}=\{V_i,w_i\}_{i\in I}$ is a $(C,C')$-controlled  Bessel $K-$fusion sequence for $\mathcal{H}$ and a $(C,C')$-controlled  $K-$fusion frame  for $\mathcal{R}_{K}$.
\end{prop}
\begin{proof}
Let $f\in \mathcal{H}$. Considering that the family $\mathcal{W}=\{W_i,w_i\}_{i\in I}$ is a $(C,C)$-controlled  $K-$fusion frame  for $\mathcal{H}$, we have \begin{align*}
 A_{CC'}\Vert K^* f\Vert^2\leq \sum_{i\in I} w_i^2  \langle C'^*\pi_{W_i}C f,f\rangle \leq B_{CC'} \Vert f\Vert^2.
\end{align*}
Firstly, let us prove that  $\{V_i,w_i\}_{i\in I}$ is a  $(C,C')$-controlled  Bessel $K-$fusion sequence for $\mathcal{H}$. We have \begin{eqnarray*}
\sum_{i\in I} w_i^2  \langle C'^*\pi_{V_i}C f,f\rangle &=& \sum_{i\in I}  w_i^2 \langle C'^*( \pi_{V_i}-\pi_{W_i})C f, f\rangle+\sum_{i\in I} w_i^2  \langle C'^*\pi_{W_i}C f,f\rangle\\
&\leq & R \Vert f\Vert^2+ B_{CC'}\Vert f\Vert^2,
\end{eqnarray*}
consequently,
 \begin{align*}
 \sum_{i\in I} w_i^2  \langle C'^*\pi_{V_i}C f,f\rangle \leq (R+ B_{CC'})\Vert f\Vert^2.
 \end{align*}
 Now, let us establish for $\{V_i,w_i\}_{i\in I}$  the left-hand side. We obtain
 \begin{align*}
 \sum_{i\in I} w_i^2  \langle C'^*\pi_{V_i}C f,f\rangle =\sum_{i\in I} w_i^2  \langle C'^*\pi_{W_i}C f,f\rangle -\sum_{i\in I}  w_i^2 \langle C'^*( \pi_{V_i}-\pi_{W_i})C f, f\rangle
 \end{align*}
 \begin{align}\label{proof: perturbation2}
\geq  A_{CC'}\Vert K^* f\Vert^2-R\Vert  f\Vert^2.
 \end{align}
  Therfore, for any $f\in\mathcal{R}_K $, we have
\begin{eqnarray*}
\Vert f\Vert =\Vert (K^\dagger\mid_{\mathcal{R}_K})^*K^* f\Vert\leq \Vert K^\dagger\Vert \Vert K^* f\Vert,
\end{eqnarray*}
that is, $\Vert K^* f\Vert^2\geq \Vert K^\dagger\Vert^{-2}\Vert \Vert f \Vert^2$. \begin{align}\label{proof: perturbation1}
-R \Vert f \Vert^2 \geq -R \Vert K^\dagger\Vert^{-2} \Vert K^* f\Vert^2.
\end{align} 
Then, according to \ref{proof: perturbation2} and \ref{proof: perturbation1}, we obtain
\begin{align*}
\sum_{i\in I} w_i^2  \langle C'^*\pi_{V_i}C f,f\rangle\geq (A_{CC'} -R \Vert K^\dagger\Vert^{-2})\Vert K^* f\Vert^2.
\end{align*}
Which completes the proof.
\end{proof}
  

  \bigskip

\noindent{ \bf Aknowlegements : } 
{\it The authors would like to thank the anonymous reviewers whose comments helped us to improve the presentation of this paper.

\smallskip

\smallskip

}


 \end{document}